\documentclass{amsart}
\usepackage{amsmath}
\usepackage{amsfonts}

\newtheorem{theorem}{Theorem}
\theoremstyle{plain}

\newtheorem{corollary}{Corollary}

\newtheorem{definition}{Definition}

\newtheorem{lemma}{Lemma}

\newtheorem{proposition}{Proposition}

\numberwithin{equation}{section}
\input{tcilatex}

\begin{document}
\title[On $s$-Geometrically Convex Functions]{On Some New Hermite-Hadamard
type inequalities for $s$-geometrically convex functions}
\author{\.{I}mdat \.{I}\c{s}can}
\address{Department of Mathematics, Faculty of Arts and Sciences,\\
Giresun University, Giresun, Turkey.}
\email{imdati@yahoo.com, imdat.iscan@giresun.edu.tr}

\begin{abstract}
In this paper, some new integral inequalities of Hermite-Hadamard type
related to the $s$-geometrically convex functions are established and some
applications to special means of positive real numbers are also given.
\end{abstract}

\maketitle

\section{Introduction}

In this section, we firstly list several definitions and some known results.

\begin{definition}
Let $I$ be an interval in $%
%TCIMACRO{\U{211d} }%
%BeginExpansion
\mathbb{R}
%EndExpansion
$. Then $f:I\rightarrow 
%TCIMACRO{\U{211d} }%
%BeginExpansion
\mathbb{R}
%EndExpansion
$ is said to be convex if%
\begin{equation*}
f\left( \lambda x+(1-\lambda )y\right) \leq \lambda f(x)+(1-\lambda )f(y)
\end{equation*}%
for all $x,y\in I$ and $\lambda \in \lbrack 0,1]$.
\end{definition}

\begin{definition}
\cite{HM94}Let $s\in (0,1]$. A function $f:I\subset \left[ 0,\infty \right)
\rightarrow \left[ 0,\infty \right) $ is said to be $s$-convex in the second
sense if%
\begin{equation*}
f\left( \lambda x+(1-\lambda )y\right) \leq \lambda ^{s}f(x)+(1-\lambda
)^{s}f(y)
\end{equation*}%
for all $x,y\in I$ and $\lambda \in \lbrack 0,1]$.
\end{definition}

\begin{definition}[\protect\cite{ZJQ12}]
A function $f:I\subset 
%TCIMACRO{\U{211d} }%
%BeginExpansion
\mathbb{R}
%EndExpansion
_{+}=\left( 0,\infty \right) \rightarrow 
%TCIMACRO{\U{211d} }%
%BeginExpansion
\mathbb{R}
%EndExpansion
_{+}$ is said to be a geometrically convex function if%
\begin{equation*}
f\left( x^{\lambda }y^{1-\lambda }\right) \leq f(x)^{\lambda
}f(y)^{1-\lambda }
\end{equation*}%
for $x,y\in I$ and $\lambda \in \lbrack 0,1]$.
\end{definition}

\begin{definition}[\protect\cite{ZJQ12}]
A function $f:I\subset 
%TCIMACRO{\U{211d} }%
%BeginExpansion
\mathbb{R}
%EndExpansion
_{+}=\left( 0,\infty \right) \rightarrow 
%TCIMACRO{\U{211d} }%
%BeginExpansion
\mathbb{R}
%EndExpansion
_{+}$ is said to be a $s$-geometrically convex function if%
\begin{equation*}
f\left( x^{\lambda }y^{1-\lambda }\right) \leq f(x)^{\lambda
^{s}}f(y)^{\left( 1-\lambda \right) ^{s}}
\end{equation*}%
for some $s\in \left( 0,1\right] $, where $x,y\in I$ and $\lambda \in
\lbrack 0,1]$.
\end{definition}

Let $f:I\subseteq \mathbb{R\rightarrow R}$ be a convex function defined on
the interval $I$ of real numbers and $a,b\in I$ with $a<b$. The following
double inequality is well known in the literature as Hermite-Hadamard
integral inequality 
\begin{equation}
f\left( \frac{a+b}{2}\right) \leq \frac{1}{b-a}\dint\limits_{a}^{b}f(x)dx%
\leq \frac{f(a)+f(b)}{2}\text{.}  \label{1-1}
\end{equation}%
Recently, several integral inequalities connected with the inequalities (\ref%
{1-1}) for the $s$-convex functions have been established by many authors
for example see \cite{ADK11,HBI09,HM94,I13,KBO07}. In \cite{ZJQ12}, The
authors has established some integral inequalities connected with the
inequalities (\ref{1-1}) for the $s$-geometrically convex and monotonically
decreasing functions. In \cite{T12}, Tunc has established inequalities for $%
s $-geometrically and geometrically convex functions which are connected
with the famous Hermite Hadamard inequality holding for convex functions. In 
\cite{T12}, Tunc also has given the following result for geometrically
convex and monotonically decreasing functions:

\begin{corollary}
Let $f:I\subset 
%TCIMACRO{\U{211d} }%
%BeginExpansion
\mathbb{R}
%EndExpansion
_{+}\rightarrow 
%TCIMACRO{\U{211d} }%
%BeginExpansion
\mathbb{R}
%EndExpansion
_{+}$ be geometrically convex and monotonically decreasing on $\left[ a,b%
\right] $, then one has%
\begin{equation}
f^{2}\left( \sqrt{ab}\right) \leq \frac{1}{\ln b-\ln a}\dint\limits_{a}^{b}%
\frac{f(x)}{x}f\left( \frac{ab}{x}\right) dx\leq f(a)f(b).  \label{1-2}
\end{equation}
\end{corollary}

Note that, the inequalities (\ref{1-2}) are also true without the condition
monotonically decreasing and the inequalities (\ref{1-2}) are sharp.

In this paper, the author give new identities for differentiable functions.
A consequence of the identities is that the author establish some new
inequalities connected with the inequalities (\ref{1-2}) for the $s$%
-geometrically convex functions.

\section{Main Results}

In order to prove our results, we need the following lemma:

\begin{lemma}
\label{2.1}Let $f:I\subseteq \mathbb{R}_{+}\mathbb{\rightarrow R}$ be
differentiable on $I^{\circ }$, and $a,b\in I$ with $a<b.$ If $f^{\prime
}\in L\left[ a,b\right] ,$ then%
\begin{eqnarray}
&&f(a)f(b)-\frac{1}{\ln b-\ln a}\dint\limits_{a}^{b}\frac{f(x)}{x}f\left( 
\frac{ab}{x}\right) dx  \notag \\
&=&\dint\limits_{0}^{1}\frac{b}{2}\ln \left( \frac{a}{b}\right) \left(
t-1\right) \left( \frac{a}{b}\right) ^{\frac{t}{2}}f\left( a^{1-t}\left(
ab\right) ^{\frac{t}{2}}\right) f^{\prime }\left( b^{1-t}\left( ab\right) ^{%
\frac{t}{2}}\right)  \notag \\
&&+\frac{a}{2}\ln \left( \frac{b}{a}\right) \left( t-1\right) \left( \frac{b%
}{a}\right) ^{\frac{t}{2}}f^{\prime }\left( a^{1-t}\left( ab\right) ^{\frac{t%
}{2}}\right) f\left( b^{1-t}\left( ab\right) ^{\frac{t}{2}}\right) dt,
\label{2-1}
\end{eqnarray}%
\begin{eqnarray}
&&f^{2}\left( \sqrt{ab}\right) -\frac{1}{\ln b-\ln a}\dint\limits_{a}^{b}%
\frac{f(x)}{x}f\left( \frac{ab}{x}\right) dx  \notag \\
&=&\dint\limits_{0}^{1}\frac{b}{2}\ln \left( \frac{a}{b}\right) t\left( 
\frac{a}{b}\right) ^{\frac{t}{2}}f\left( a^{1-t}\left( ab\right) ^{\frac{t}{2%
}}\right) f^{\prime }\left( b^{1-t}\left( ab\right) ^{\frac{t}{2}}\right) 
\notag \\
&&+\frac{a}{2}\ln \left( \frac{b}{a}\right) t\left( \frac{b}{a}\right) ^{%
\frac{t}{2}}f^{\prime }\left( a^{1-t}\left( ab\right) ^{\frac{t}{2}}\right)
f\left( b^{1-t}\left( ab\right) ^{\frac{t}{2}}\right) dt.  \label{2-2}
\end{eqnarray}
\end{lemma}

\begin{proof}
Integrating by part and changing variables of integration yields%
\begin{eqnarray*}
&&\dint\limits_{0}^{1}\frac{b}{2}\ln \left( \frac{a}{b}\right) \left(
t-1\right) \left( \frac{a}{b}\right) ^{\frac{t}{2}}f\left( a^{1-t}\left(
ab\right) ^{\frac{t}{2}}\right) f^{\prime }\left( b^{1-t}\left( ab\right) ^{%
\frac{t}{2}}\right) \\
&&+\frac{a}{2}\ln \left( \frac{b}{a}\right) \left( t-1\right) \left( \frac{b%
}{a}\right) ^{\frac{t}{2}}f^{\prime }\left( a^{1-t}\left( ab\right) ^{\frac{t%
}{2}}\right) f\left( b^{1-t}\left( ab\right) ^{\frac{t}{2}}\right) dt \\
&=&\dint\limits_{0}^{1}\left( t-1\right) \left[ f\left( a^{1-t}\left(
ab\right) ^{\frac{t}{2}}\right) f\left( b^{1-t}\left( ab\right) ^{\frac{t}{2}%
}\right) \right] ^{\prime }dt \\
&=&\left. \left( t-1\right) f\left( a^{1-t}\left( ab\right) ^{\frac{t}{2}%
}\right) f\left( b^{1-t}\left( ab\right) ^{\frac{t}{2}}\right) \right\vert
_{0}^{1}-\dint\limits_{0}^{1}f\left( a^{1-t}\left( ab\right) ^{\frac{t}{2}%
}\right) f\left( b^{1-t}\left( ab\right) ^{\frac{t}{2}}\right) dt \\
&=&f(a)f(b)-\frac{2}{\ln b-\ln a}\dint\limits_{a}^{\sqrt{ab}}\frac{f(x)}{x}%
f\left( \frac{ab}{x}\right) dx.
\end{eqnarray*}%
By the following equality, we obtain the inequality (\ref{2-1})%
\begin{equation*}
\dint\limits_{a}^{\sqrt{ab}}\frac{f(x)}{x}f\left( \frac{ab}{x}\right)
dx=\dint\limits_{\sqrt{ab}}^{b}\frac{f(x)}{x}f\left( \frac{ab}{x}\right) dx.
\end{equation*}%
\begin{eqnarray*}
&&\dint\limits_{0}^{1}\frac{b}{2}\ln \left( \frac{a}{b}\right) t\left( \frac{%
a}{b}\right) ^{\frac{t}{2}}f\left( a^{1-t}\left( ab\right) ^{\frac{t}{2}%
}\right) f^{\prime }\left( b^{1-t}\left( ab\right) ^{\frac{t}{2}}\right) \\
&&+\frac{a}{2}\ln \left( \frac{b}{a}\right) t\left( \frac{b}{a}\right) ^{%
\frac{t}{2}}f^{\prime }\left( a^{1-t}\left( ab\right) ^{\frac{t}{2}}\right)
f\left( b^{1-t}\left( ab\right) ^{\frac{t}{2}}\right) dt \\
&=&\dint\limits_{0}^{1}t\left[ f\left( a^{1-t}\left( ab\right) ^{\frac{t}{2}%
}\right) f\left( b^{1-t}\left( ab\right) ^{\frac{t}{2}}\right) \right]
^{\prime }dt \\
&=&\left. tf\left( a^{1-t}\left( ab\right) ^{\frac{t}{2}}\right) f\left(
b^{1-t}\left( ab\right) ^{\frac{t}{2}}\right) \right\vert
_{0}^{1}-\dint\limits_{0}^{1}f\left( a^{1-t}\left( ab\right) ^{\frac{t}{2}%
}\right) f\left( b^{1-t}\left( ab\right) ^{\frac{t}{2}}\right) dt \\
&=&f^{2}\left( \sqrt{ab}\right) -\frac{2}{\ln b-\ln a}\dint\limits_{a}^{%
\sqrt{ab}}\frac{f(x)}{x}f\left( \frac{ab}{x}\right) dx \\
&=&f^{2}\left( \sqrt{ab}\right) -\frac{1}{\ln b-\ln a}\dint\limits_{a}^{b}%
\frac{f(x)}{x}f\left( \frac{ab}{x}\right) dx.
\end{eqnarray*}%
This completes the proof of Lemma \ref{2.1}.
\end{proof}

\begin{theorem}
\label{2.2}Let $f:I\subseteq \mathbb{R}_{+}\mathbb{\rightarrow R}_{+}$ be
differentiable on $I^{\circ }$, and $a,b\in I^{\circ }$ with $a<b$ and $%
f^{\prime }\in L\left[ a,b\right] .$ If $\left\vert f^{\prime }\right\vert
^{q}$ is $s$-geometrically convex on $\left[ a,b\right] $ for $q\geq 1$ and $%
s\in \left( 0,1\right] ,$ then%
\begin{equation}
\left\vert f(a)f(b)-\frac{1}{\ln b-\ln a}\dint\limits_{a}^{b}\frac{f(x)}{x}%
f\left( \frac{ab}{x}\right) dx\right\vert \leq \ln \left( \frac{b}{a}\right)
\left( \frac{1}{2}\right) ^{2-\frac{1}{q}}H_{1}\left( s,q;h_{1}(\theta
),h_{1}(\vartheta )\right) ,  \label{2-3}
\end{equation}%
\begin{equation}
\left\vert f^{2}\left( \sqrt{ab}\right) -\frac{1}{\ln b-\ln a}%
\dint\limits_{a}^{b}\frac{f(x)}{x}f\left( \frac{ab}{x}\right) dx\right\vert
\leq \ln \left( \frac{b}{a}\right) \left( \frac{1}{2}\right) ^{2-\frac{1}{q}%
}H_{2}\left( s,q;h_{2}(\theta ),h_{2}(\vartheta )\right) ,  \label{2-4}
\end{equation}%
where $M_{1}=\underset{x\in \left[ a,\sqrt{ab}\right] }{\max \left\vert
f(x)\right\vert },\ M_{2}=\underset{x\in \left[ \sqrt{ab},b\right] }{\max
\left\vert f(x)\right\vert },$%
\begin{equation}
h_{1}(u)=\left\{ 
\begin{array}{c}
\frac{1}{2},\ \ \ \ \ \ \ \ \ \ \ u=1 \\ 
\frac{u-\ln u-1}{\left( \ln u\right) ^{2}},\ u\neq 1%
\end{array}%
\right. ,\ h_{2}(u)=\left\{ 
\begin{array}{c}
\frac{1}{2},\ \ \ \ \ \ \ \ \ \ \ u=1 \\ 
\frac{u\ln u-u+1}{\left( \ln u\right) ^{2}},\ u\neq 1%
\end{array}%
\right. ,\ u>0  \label{2-5}
\end{equation}%
\begin{equation}
\theta =\left( \frac{a\left\vert f^{\prime }(a)\right\vert ^{s}}{b\left\vert
f^{\prime }(b)\right\vert ^{s}}\right) ^{q/2},\ \vartheta =\left( \frac{%
b\left\vert f^{\prime }(b)\right\vert ^{s}}{a\left\vert f^{\prime
}(a)\right\vert ^{s}}\right) ^{q/2},\   \label{2-6}
\end{equation}%
\begin{eqnarray}
&&H_{i}\left( s,q;h_{i}(\theta ),h_{i}(\vartheta )\right)  \label{2-7} \\
&=&\left\{ 
\begin{array}{c}
b\left\vert f^{\prime }(b)\right\vert ^{s}M_{1}h_{i}^{1/q}\left( \theta
\right) +a\left\vert f^{\prime }(a)\right\vert ^{s}M_{2}h_{i}^{1/q}\left(
\vartheta \right) , \\ 
\ \left\vert f^{\prime }(a)\right\vert ,\ \left\vert f^{\prime
}(b)\right\vert \leq 1, \\ 
b\left\vert f^{\prime }\left( b\right) \right\vert \left\vert f^{\prime
}\left( a\right) \right\vert ^{1-s}M_{1}h_{i}^{1/q}\left( \theta \right)
+a\left\vert f^{\prime }\left( a\right) \right\vert \left\vert f^{\prime
}\left( b\right) \right\vert ^{1-s}M_{2}h_{i}^{1/q}\left( \vartheta \right) ,
\\ 
\ \left\vert f^{\prime }(a)\right\vert ,\ \left\vert f^{\prime
}(b)\right\vert \geq 1, \\ 
b\left\vert f^{\prime }\left( b\right) \right\vert M_{1}h_{i}^{1/q}\left(
\theta \right) +a\left\vert f^{\prime }\left( a\right) \right\vert
^{s}\left\vert f^{\prime }\left( b\right) \right\vert
^{1-s}M_{2}h_{i}^{1/q}\left( \vartheta \right) , \\ 
\ \left\vert f^{\prime }(a)\right\vert \leq 1\leq \ \left\vert f^{\prime
}(b)\right\vert , \\ 
b\left\vert f^{\prime }\left( b\right) \right\vert ^{s}\left\vert f^{\prime
}\left( a\right) \right\vert ^{1-s}M_{1}h_{i}^{1/q}\left( \theta \right)
+a\left\vert f^{\prime }(a)\right\vert M_{2}h_{i}^{1/q}\left( \vartheta
\right) , \\ 
\ \left\vert f^{\prime }(b)\right\vert \leq 1\leq \ \left\vert f^{\prime
}(a)\right\vert .%
\end{array}%
\right. ,i=1,2  \notag
\end{eqnarray}
\end{theorem}

\bigskip

\begin{proof}
(1) Let $M_{1}=\underset{x\in \left[ a,\sqrt{ab}\right] }{\max \left\vert
f(x)\right\vert },\ M_{2}=\underset{x\in \left[ \sqrt{ab},b\right] }{\max
\left\vert f(x)\right\vert }.$ Since $\left\vert f^{\prime }\right\vert ^{q}$
is $s$-geometrically convex on $\left[ a,b\right] $, from lemma \ref{2.1}
and power mean inequality, we have%
\begin{eqnarray*}
&&\left\vert f(a)f(b)-\frac{1}{\ln b-\ln a}\dint\limits_{a}^{b}\frac{f(x)}{x}%
f\left( \frac{ab}{x}\right) dx\right\vert \\
&\leq &\dint\limits_{0}^{1}\frac{b}{2}\left\vert \ln \left( \frac{a}{b}%
\right) \right\vert \left\vert t-1\right\vert \left( \frac{a}{b}\right) ^{%
\frac{t}{2}}\left\vert f\left( a^{1-t}\left( ab\right) ^{\frac{t}{2}}\right)
\right\vert \left\vert f^{\prime }\left( b^{1-t}\left( ab\right) ^{\frac{t}{2%
}}\right) \right\vert \\
&&+\frac{a}{2}\ln \left( \frac{b}{a}\right) \left\vert t-1\right\vert \left( 
\frac{b}{a}\right) ^{\frac{t}{2}}\left\vert f^{\prime }\left( a^{1-t}\left(
ab\right) ^{\frac{t}{2}}\right) \right\vert \left\vert f\left( b^{1-t}\left(
ab\right) ^{\frac{t}{2}}\right) \right\vert dt
\end{eqnarray*}%
\begin{eqnarray*}
&\leq &\frac{b}{2}\left\vert \ln \left( \frac{a}{b}\right) \right\vert
M_{1}\dint\limits_{0}^{1}\left( 1-t\right) \left( \frac{a}{b}\right) ^{\frac{%
t}{2}}\left\vert f^{\prime }\left( b^{1-t}\left( ab\right) ^{\frac{t}{2}%
}\right) \right\vert dt \\
&&+\frac{a}{2}\ln \left( \frac{b}{a}\right) M_{2}\dint\limits_{0}^{1}\left(
1-t\right) \left( \frac{b}{a}\right) ^{\frac{t}{2}}\left\vert f^{\prime
}\left( a^{1-t}\left( ab\right) ^{\frac{t}{2}}\right) \right\vert dt
\end{eqnarray*}%
\begin{eqnarray*}
&\leq &\frac{b}{2}\left\vert \ln \left( \frac{a}{b}\right) \right\vert
M_{1}\left( \dint\limits_{0}^{1}\left( 1-t\right) dt\right) ^{1-\frac{1}{q}%
}\left( \dint\limits_{0}^{1}\left( 1-t\right) \left( \frac{a}{b}\right) ^{%
\frac{qt}{2}}\left\vert f^{\prime }\left( b^{1-t}\left( ab\right) ^{\frac{t}{%
2}}\right) \right\vert ^{q}dt\right) ^{\frac{1}{q}} \\
&&+\frac{a}{2}\ln \left( \frac{b}{a}\right) M_{2}\left(
\dint\limits_{0}^{1}\left( 1-t\right) dt\right) ^{1-\frac{1}{q}}\left(
\dint\limits_{0}^{1}\left( 1-t\right) \left( \frac{b}{a}\right) ^{\frac{qt}{2%
}}\left\vert f^{\prime }\left( a^{1-t}\left( ab\right) ^{\frac{t}{2}}\right)
\right\vert ^{q}dt\right) ^{\frac{1}{q}}
\end{eqnarray*}%
\begin{eqnarray}
&\leq &\frac{b\ln \left( \frac{b}{a}\right) M_{1}}{2}\left( \frac{1}{2}%
\right) ^{1-\frac{1}{q}}\left( \dint\limits_{0}^{1}\left( 1-t\right) \left( 
\frac{a}{b}\right) ^{\frac{qt}{2}}\left\vert f^{\prime }\left( a\right)
\right\vert ^{q\left( t/2\right) ^{s}}\left\vert f^{\prime }\left( b\right)
\right\vert ^{q\left( \left( 2-t\right) /2\right) ^{s}}dt\right) ^{\frac{1}{q%
}}  \notag \\
&&+\frac{a\ln \left( \frac{b}{a}\right) M_{2}}{2}\left( \frac{1}{2}\right)
^{1-\frac{1}{q}}\left( \dint\limits_{0}^{1}\left( 1-t\right) \left( \frac{b}{%
a}\right) ^{\frac{qt}{2}}\left\vert f^{\prime }\left( b\right) \right\vert
^{q\left( t/2\right) ^{s}}\left\vert f^{\prime }\left( a\right) \right\vert
^{q\left( \left( 2-t\right) /2\right) ^{s}}dt\right) ^{\frac{1}{q}}.
\label{2-8}
\end{eqnarray}%
If $0<\mu \leq 1\leq \eta ,\ 0<t,s\leq 1,$ then%
\begin{equation}
\mu ^{t^{s}}\leq \mu ^{ts},\ \ \eta ^{t^{s}}\leq \eta ^{ts+1-s}.  \label{2-a}
\end{equation}

(i) \ If $1\geq \left\vert f^{\prime }(a)\right\vert ,\ \left\vert f^{\prime
}(b)\right\vert ,$ by (\ref{2-a}) we obtain that%
\begin{eqnarray*}
&&\dint\limits_{0}^{1}\left( 1-t\right) \left( \frac{a}{b}\right) ^{\frac{qt%
}{2}}\left\vert f^{\prime }\left( a\right) \right\vert ^{q\left( t/2\right)
^{s}}\left\vert f^{\prime }\left( b\right) \right\vert ^{q\left( \left(
2-t\right) /2\right) ^{s}}dt \\
&\leq &\dint\limits_{0}^{1}\left( 1-t\right) \left( \frac{a}{b}\right) ^{%
\frac{qt}{2}}\left\vert f^{\prime }\left( a\right) \right\vert ^{\frac{qst}{2%
}}\left\vert f^{\prime }\left( b\right) \right\vert ^{q\frac{qs\left(
2-t\right) }{2}}dt=\left\vert f^{\prime }\left( b\right) \right\vert
^{qs}h_{1}\left( \theta \right) ,
\end{eqnarray*}%
\begin{eqnarray}
&&\dint\limits_{0}^{1}\left( 1-t\right) \left( \frac{b}{a}\right) ^{\frac{qt%
}{2}}\left\vert f^{\prime }\left( b\right) \right\vert ^{q\left( t/2\right)
^{s}}\left\vert f^{\prime }\left( a\right) \right\vert ^{q\left( \left(
2-t\right) /2\right) ^{s}}dt  \label{2-9} \\
&\leq &\dint\limits_{0}^{1}\left( 1-t\right) \left( \frac{b}{a}\right) ^{%
\frac{qt}{2}}\left\vert f^{\prime }\left( b\right) \right\vert ^{\frac{qst}{2%
}}\left\vert f^{\prime }\left( a\right) \right\vert ^{\frac{qs\left(
2-t\right) }{2}}dt=\left\vert f^{\prime }\left( a\right) \right\vert
^{qs}h_{1}\left( \vartheta \right) .  \notag
\end{eqnarray}

(ii) \ If $1\leq \left\vert f^{\prime }(a)\right\vert ,\ \left\vert
f^{\prime }(b)\right\vert ,$ by (\ref{2-a}) we obtain that%
\begin{eqnarray*}
&&\dint\limits_{0}^{1}\left( 1-t\right) \left( \frac{a}{b}\right) ^{\frac{qt%
}{2}}\left\vert f^{\prime }\left( a\right) \right\vert ^{q\left( t/2\right)
^{s}}\left\vert f^{\prime }\left( b\right) \right\vert ^{q\left( \left(
2-t\right) /2\right) ^{s}}dt \\
&\leq &\dint\limits_{0}^{1}\left( 1-t\right) \left( \frac{a}{b}\right) ^{%
\frac{qt}{2}}\left\vert f^{\prime }\left( a\right) \right\vert ^{q\left( 
\frac{st}{2}+1-s\right) }\left\vert f^{\prime }\left( b\right) \right\vert
^{q\left( 1-\frac{st}{2}\right) }dt \\
&=&\left( \left\vert f^{\prime }\left( a\right) \right\vert ^{1-s}\left\vert
f^{\prime }\left( b\right) \right\vert \right) ^{q}h_{1}\left( \theta
\right) ,
\end{eqnarray*}%
\begin{eqnarray}
&&\dint\limits_{0}^{1}\left( 1-t\right) \left( \frac{b}{a}\right) ^{\frac{qt%
}{2}}\left\vert f^{\prime }\left( b\right) \right\vert ^{q\left( t/2\right)
^{s}}\left\vert f^{\prime }\left( a\right) \right\vert ^{q\left( \left(
2-t\right) /2\right) ^{s}}dt  \label{2-10} \\
&\leq &\dint\limits_{0}^{1}\left( 1-t\right) \left( \frac{b}{a}\right) ^{%
\frac{qt}{2}}\left\vert f^{\prime }\left( b\right) \right\vert ^{q\left( 
\frac{st}{2}+1-s\right) }\left\vert f^{\prime }\left( a\right) \right\vert
^{q\left( 1-\frac{st}{2}\right) }dt  \notag \\
&=&\left( \left\vert f^{\prime }\left( b\right) \right\vert ^{1-s}\left\vert
f^{\prime }\left( a\right) \right\vert \right) ^{q}h_{1}\left( \vartheta
\right) .  \notag
\end{eqnarray}%
(iii) \ If $\left\vert f^{\prime }(a)\right\vert \leq 1\leq \ \left\vert
f^{\prime }(b)\right\vert ,$ by (\ref{2-a}) we obtain that%
\begin{eqnarray*}
&&\dint\limits_{0}^{1}\left( 1-t\right) \left( \frac{a}{b}\right) ^{\frac{qt%
}{2}}\left\vert f^{\prime }\left( a\right) \right\vert ^{q\left( t/2\right)
^{s}}\left\vert f^{\prime }\left( b\right) \right\vert ^{q\left( \left(
2-t\right) /2\right) ^{s}}dt \\
&\leq &\dint\limits_{0}^{1}\left( 1-t\right) \left( \frac{a}{b}\right) ^{%
\frac{qt}{2}}\left\vert f^{\prime }\left( a\right) \right\vert ^{\frac{qst}{2%
}}\left\vert f^{\prime }\left( b\right) \right\vert ^{q\left( 1-\frac{st}{2}%
\right) }dt=\left\vert f^{\prime }\left( b\right) \right\vert
^{q}h_{1}\left( \theta \right) ,
\end{eqnarray*}%
\begin{eqnarray}
&&\dint\limits_{0}^{1}\left( 1-t\right) \left( \frac{b}{a}\right) ^{\frac{qt%
}{2}}\left\vert f^{\prime }\left( b\right) \right\vert ^{q\left( t/2\right)
^{s}}\left\vert f^{\prime }\left( a\right) \right\vert ^{q\left( \left(
2-t\right) /2\right) ^{s}}dt  \label{2-11} \\
&\leq &\dint\limits_{0}^{1}\left( 1-t\right) \left( \frac{b}{a}\right) ^{%
\frac{qt}{2}}\left\vert f^{\prime }\left( b\right) \right\vert ^{q\left( 
\frac{st}{2}+1-s\right) }\left\vert f^{\prime }\left( a\right) \right\vert ^{%
\frac{qs\left( 2-t\right) }{2}}dt  \notag \\
&=&\left( \left\vert f^{\prime }\left( a\right) \right\vert ^{s}\left\vert
f^{\prime }\left( b\right) \right\vert ^{1-s}\right) ^{q}h_{1}\left(
\vartheta \right) .  \notag
\end{eqnarray}%
(iv) \ If $\left\vert f^{\prime }(b)\right\vert \leq 1\leq \ \left\vert
f^{\prime }(a)\right\vert ,$ by (\ref{2-a}) we obtain that%
\begin{eqnarray*}
&&\dint\limits_{0}^{1}\left( 1-t\right) \left( \frac{a}{b}\right) ^{\frac{qt%
}{2}}\left\vert f^{\prime }\left( a\right) \right\vert ^{q\left( t/2\right)
^{s}}\left\vert f^{\prime }\left( b\right) \right\vert ^{q\left( \left(
2-t\right) /2\right) ^{s}}dt \\
&\leq &\dint\limits_{0}^{1}\left( 1-t\right) \left( \frac{a}{b}\right) ^{%
\frac{qt}{2}}\left\vert f^{\prime }\left( a\right) \right\vert ^{q\left( 
\frac{st}{2}+1-s\right) }\left\vert f^{\prime }\left( b\right) \right\vert ^{%
\frac{qs\left( 2-t\right) }{2}}dt \\
&=&\left( \left\vert f^{\prime }\left( a\right) \right\vert ^{1-s}\left\vert
f^{\prime }\left( b\right) \right\vert ^{s}\right) ^{q}h_{1}\left( \theta
\right) ,
\end{eqnarray*}%
\begin{eqnarray}
&&\dint\limits_{0}^{1}\left( 1-t\right) \left( \frac{b}{a}\right) ^{\frac{qt%
}{2}}\left\vert f^{\prime }\left( b\right) \right\vert ^{q\left( t/2\right)
^{s}}\left\vert f^{\prime }\left( a\right) \right\vert ^{q\left( \left(
2-t\right) /2\right) ^{s}}dt  \label{2-12} \\
&\leq &\dint\limits_{0}^{1}\left( 1-t\right) \left( \frac{b}{a}\right) ^{%
\frac{qt}{2}}\left\vert f^{\prime }\left( b\right) \right\vert ^{\frac{qst}{2%
}}\left\vert f^{\prime }\left( a\right) \right\vert ^{q\left( 1-\frac{st}{2}%
\right) }dt=\left\vert f^{\prime }\left( a\right) \right\vert
^{q}h_{1}\left( \vartheta \right) .  \notag
\end{eqnarray}%
From (\ref{2-8}) to (\ref{2-12}), (\ref{2-3}) holds.

(2) Let $M_{1}=\underset{x\in \left[ a,\sqrt{ab}\right] }{\max \left\vert
f(x)\right\vert },\ M_{2}=\underset{x\in \left[ \sqrt{ab},b\right] }{\max
\left\vert f(x)\right\vert }.$ Since $\left\vert f^{\prime }\right\vert ^{q}$
is $s$-geometrically convex on $\left[ a,b\right] $, from lemma \ref{2.1}
and H\"{o}lder inequality, we have%
\begin{eqnarray*}
&&\left\vert f^{2}\left( \sqrt{ab}\right) -\frac{1}{\ln b-\ln a}%
\dint\limits_{a}^{b}\frac{f(x)}{x}f\left( \frac{ab}{x}\right) dx\right\vert
\\
&\leq &\dint\limits_{0}^{1}\frac{b}{2}\left\vert \ln \left( \frac{a}{b}%
\right) \right\vert t\left( \frac{a}{b}\right) ^{\frac{t}{2}}\left\vert
f\left( a^{1-t}\left( ab\right) ^{\frac{t}{2}}\right) \right\vert \left\vert
f^{\prime }\left( b^{1-t}\left( ab\right) ^{\frac{t}{2}}\right) \right\vert
\\
&&+\frac{a}{2}\ln \left( \frac{b}{a}\right) t\left( \frac{b}{a}\right) ^{%
\frac{t}{2}}\left\vert f^{\prime }\left( a^{1-t}\left( ab\right) ^{\frac{t}{2%
}}\right) \right\vert \left\vert f\left( b^{1-t}\left( ab\right) ^{\frac{t}{2%
}}\right) \right\vert dt
\end{eqnarray*}%
\begin{eqnarray*}
&\leq &\frac{b}{2}\ln \left( \frac{b}{a}\right)
M_{1}\dint\limits_{0}^{1}t\left( \frac{a}{b}\right) ^{\frac{t}{2}}\left\vert
f^{\prime }\left( b^{1-t}\left( ab\right) ^{\frac{t}{2}}\right) \right\vert
dt \\
&&+\frac{a}{2}\ln \left( \frac{b}{a}\right) M_{2}\dint\limits_{0}^{1}t\left( 
\frac{b}{a}\right) ^{\frac{t}{2}}\left\vert f^{\prime }\left( a^{1-t}\left(
ab\right) ^{\frac{t}{2}}\right) \right\vert dt
\end{eqnarray*}%
\begin{eqnarray*}
&\leq &\frac{b}{2}\ln \left( \frac{b}{a}\right) M_{1}\left(
\dint\limits_{0}^{1}tdt\right) ^{1-\frac{1}{q}}\left(
\dint\limits_{0}^{1}t\left( \frac{a}{b}\right) ^{\frac{qt}{2}}\left\vert
f^{\prime }\left( b^{1-t}\left( ab\right) ^{\frac{t}{2}}\right) \right\vert
^{q}dt\right) ^{\frac{1}{q}} \\
&&+\frac{a}{2}\ln \left( \frac{b}{a}\right) M_{2}\left(
\dint\limits_{0}^{1}tdt\right) ^{1-\frac{1}{q}}\left(
\dint\limits_{0}^{1}t\left( \frac{b}{a}\right) ^{\frac{qt}{2}}\left\vert
f^{\prime }\left( a^{1-t}\left( ab\right) ^{\frac{t}{2}}\right) \right\vert
^{q}dt\right) ^{\frac{1}{q}}
\end{eqnarray*}%
\begin{eqnarray}
&\leq &\frac{b\ln \left( \frac{b}{a}\right) M_{1}}{2}\left( \frac{1}{2}%
\right) ^{1-\frac{1}{q}}\left( \dint\limits_{0}^{1}t\left( \frac{a}{b}%
\right) ^{\frac{qt}{2}}\left\vert f^{\prime }\left( a\right) \right\vert
^{q\left( t/2\right) ^{s}}\left\vert f^{\prime }\left( b\right) \right\vert
^{q\left( \left( 2-t\right) /2\right) ^{s}}dt\right) ^{\frac{1}{q}}  \notag
\\
&&+\frac{a\ln \left( \frac{b}{a}\right) M_{2}}{2}\left( \frac{1}{2}\right)
^{1-\frac{1}{q}}\left( \dint\limits_{0}^{1}t\left( \frac{b}{a}\right) ^{%
\frac{qt}{2}}\left\vert f^{\prime }\left( b\right) \right\vert ^{q\left(
t/2\right) ^{s}}\left\vert f^{\prime }\left( a\right) \right\vert ^{q\left(
\left( 2-t\right) /2\right) ^{s}}dt\right) ^{\frac{1}{q}}.  \label{2-13}
\end{eqnarray}

(i) \ If $1\geq \left\vert f^{\prime }(a)\right\vert ,\ \left\vert f^{\prime
}(b)\right\vert ,$ by (\ref{2-a}) we obtain that%
\begin{equation*}
\dint\limits_{0}^{1}t\left( \frac{a}{b}\right) ^{\frac{qt}{2}}\left\vert
f^{\prime }\left( a\right) \right\vert ^{q\left( t/2\right) ^{s}}\left\vert
f^{\prime }\left( b\right) \right\vert ^{q\left( \left( 2-t\right) /2\right)
^{s}}dt\leq \left\vert f^{\prime }\left( b\right) \right\vert
^{qs}h_{2}\left( \theta \right) ,
\end{equation*}%
\begin{equation}
\dint\limits_{0}^{1}t\left( \frac{b}{a}\right) ^{\frac{qt}{2}}\left\vert
f^{\prime }\left( b\right) \right\vert ^{q\left( t/2\right) ^{s}}\left\vert
f^{\prime }\left( a\right) \right\vert ^{q\left( \left( 2-t\right) /2\right)
^{s}}dt\leq \left\vert f^{\prime }\left( a\right) \right\vert
^{qs}h_{2}\left( \vartheta \right) .  \label{2-14}
\end{equation}

(ii) \ If $1\leq \left\vert f^{\prime }(a)\right\vert ,\ \left\vert
f^{\prime }(b)\right\vert ,$ by (\ref{2-a}) we obtain that%
\begin{equation*}
\dint\limits_{0}^{1}t\left( \frac{a}{b}\right) ^{\frac{qt}{2}}\left\vert
f^{\prime }\left( a\right) \right\vert ^{q\left( t/2\right) ^{s}}\left\vert
f^{\prime }\left( b\right) \right\vert ^{q\left( \left( 2-t\right) /2\right)
^{s}}dt\leq \left( \left\vert f^{\prime }\left( a\right) \right\vert
^{1-s}\left\vert f^{\prime }\left( b\right) \right\vert \right)
^{q}h_{2}\left( \theta \right) ,
\end{equation*}%
\begin{equation}
\dint\limits_{0}^{1}t\left( \frac{b}{a}\right) ^{\frac{qt}{2}}\left\vert
f^{\prime }\left( b\right) \right\vert ^{q\left( t/2\right) ^{s}}\left\vert
f^{\prime }\left( a\right) \right\vert ^{q\left( \left( 2-t\right) /2\right)
^{s}}dt\leq \left( \left\vert f^{\prime }\left( b\right) \right\vert
^{1-s}\left\vert f^{\prime }\left( a\right) \right\vert \right)
^{q}h_{2}\left( \vartheta \right) .  \label{2-15}
\end{equation}%
(iii) \ If $\left\vert f^{\prime }(a)\right\vert \leq 1\leq \ \left\vert
f^{\prime }(b)\right\vert ,$ by (\ref{2-a}) we obtain that%
\begin{equation*}
\dint\limits_{0}^{1}t\left( \frac{a}{b}\right) ^{\frac{qt}{2}}\left\vert
f^{\prime }\left( a\right) \right\vert ^{q\left( t/2\right) ^{s}}\left\vert
f^{\prime }\left( b\right) \right\vert ^{q\left( \left( 2-t\right) /2\right)
^{s}}dt\leq \left\vert f^{\prime }\left( b\right) \right\vert
^{q}h_{2}\left( \theta \right) ,
\end{equation*}%
\begin{equation}
\dint\limits_{0}^{1}t\left( \frac{b}{a}\right) ^{\frac{qt}{2}}\left\vert
f^{\prime }\left( b\right) \right\vert ^{q\left( t/2\right) ^{s}}\left\vert
f^{\prime }\left( a\right) \right\vert ^{q\left( \left( 2-t\right) /2\right)
^{s}}dt\leq \left( \left\vert f^{\prime }\left( a\right) \right\vert
^{s}\left\vert f^{\prime }\left( b\right) \right\vert ^{1-s}\right)
^{q}h_{2}\left( \vartheta \right) .  \label{2-16}
\end{equation}%
(iv) \ If $\left\vert f^{\prime }(b)\right\vert \leq 1\leq \ \left\vert
f^{\prime }(a)\right\vert ,$ by (\ref{2-a}) we obtain that%
\begin{equation*}
\dint\limits_{0}^{1}t\left( \frac{a}{b}\right) ^{\frac{qt}{2}}\left\vert
f^{\prime }\left( a\right) \right\vert ^{q\left( t/2\right) ^{s}}\left\vert
f^{\prime }\left( b\right) \right\vert ^{q\left( \left( 2-t\right) /2\right)
^{s}}dt\leq \left( \left\vert f^{\prime }\left( a\right) \right\vert
^{1-s}\left\vert f^{\prime }\left( b\right) \right\vert ^{s}\right)
^{q}h_{2}\left( \theta \right) ,
\end{equation*}%
\begin{equation}
\dint\limits_{0}^{1}t\left( \frac{b}{a}\right) ^{\frac{qt}{2}}\left\vert
f^{\prime }\left( b\right) \right\vert ^{q\left( t/2\right) ^{s}}\left\vert
f^{\prime }\left( a\right) \right\vert ^{q\left( \left( 2-t\right) /2\right)
^{s}}dt\leq \left\vert f^{\prime }\left( a\right) \right\vert
^{q}h_{2}\left( \vartheta \right) .  \label{2-17}
\end{equation}%
From (\ref{2-13}) to (\ref{2-17}), (\ref{2-4}) holds. This completes the
required proof.
\end{proof}

If taking $s=1$ in Theorem \ref{2.2}, we can derive the following corollary.

\begin{corollary}
Let $f:I\subseteq \mathbb{R}_{+}\mathbb{\rightarrow R}_{+}$ be
differentiable on $I^{\circ }$, and $a,b\in I^{\circ }$ with $a<b$ and $%
f^{\prime }\in L\left[ a,b\right] .$ If $\left\vert f^{\prime }\right\vert
^{q}$ is geometrically convex on $\left[ a,b\right] $ for $q\geq 1$, then%
\begin{equation*}
\left\vert f(a)f(b)-\frac{1}{\ln b-\ln a}\dint\limits_{a}^{b}\frac{f(x)}{x}%
f\left( \frac{ab}{x}\right) dx\right\vert \leq \ln \left( \frac{b}{a}\right)
\left( \frac{1}{2}\right) ^{2-\frac{1}{q}}H_{1}\left( 1,q;h_{1}(\theta
),h_{1}(\vartheta )\right) ,
\end{equation*}%
\begin{equation*}
\left\vert f^{2}\left( \sqrt{ab}\right) -\frac{1}{\ln b-\ln a}%
\dint\limits_{a}^{b}\frac{f(x)}{x}f\left( \frac{ab}{x}\right) dx\right\vert
\leq \ln \left( \frac{b}{a}\right) \left( \frac{1}{2}\right) ^{2-\frac{1}{q}%
}H_{2}\left( 1,q;h_{2}(\theta ),h_{2}(\vartheta )\right) ,
\end{equation*}%
where $\theta ,\ \vartheta $, $H_{1},H_{2},h_{1}$ and $h_{2}$ are the same
as in Theorem \ref{2.2}.
\end{corollary}

If taking $q=1$ in Theorem \ref{2.2}, we can derive the following corollary.

\begin{corollary}
Let $f:I\subseteq \mathbb{R}_{+}\mathbb{\rightarrow R}_{+}$ be
differentiable on $I^{\circ }$, and $a,b\in I^{\circ }$ with $a<b$ and $%
f^{\prime }\in L\left[ a,b\right] .$ If $\left\vert f^{\prime }\right\vert $
is $s$-geometrically convex on $\left[ a,b\right] $ for $s\in \left( 0,1%
\right] ,$ then%
\begin{equation*}
\left\vert f(a)f(b)-\frac{1}{\ln b-\ln a}\dint\limits_{a}^{b}\frac{f(x)}{x}%
f\left( \frac{ab}{x}\right) dx\right\vert \leq \ln \left( \frac{b}{a}\right)
\left( \frac{1}{2}\right) H_{1}\left( s,1;h_{1}(\theta ),h_{1}(\vartheta
)\right) ,
\end{equation*}%
\begin{equation*}
\left\vert f^{2}\left( \sqrt{ab}\right) -\frac{1}{\ln b-\ln a}%
\dint\limits_{a}^{b}\frac{f(x)}{x}f\left( \frac{ab}{x}\right) dx\right\vert
\leq \ln \left( \frac{b}{a}\right) \left( \frac{1}{2}\right) H_{2}\left(
s,1;h_{2}(\theta ),h_{2}(\vartheta )\right) ,
\end{equation*}%
where $\theta ,\ \vartheta $, $H_{1},H_{2},h_{1}$ and $h_{2}$ are the same
as in Theorem \ref{2.2}.
\end{corollary}

\begin{theorem}
\label{2.3}Let $f:I\subseteq \mathbb{R}_{+}\mathbb{\rightarrow R}_{+}$ be
differentiable on $I^{\circ }$, and $a,b\in I^{\circ }$ with $a<b$ and $%
f^{\prime }\in L\left[ a,b\right] .$ If $\left\vert f^{\prime }\right\vert
^{q}$ is $s$-geometrically convex on $\left[ a,b\right] $ for $q>1$ and $%
s\in \left( 0,1\right] ,$ then%
\begin{equation}
\left\vert f(a)f(b)-\frac{1}{\ln b-\ln a}\dint\limits_{a}^{b}\frac{f(x)}{x}%
f\left( \frac{ab}{x}\right) dx\right\vert \leq \frac{\ln \left( \frac{b}{a}%
\right) }{2}\left( \frac{q-1}{2q-1}\right) ^{1-\frac{1}{q}}H_{3}\left(
s,q;h_{3}(\theta ),h_{3}(\vartheta )\right) ,  \label{2-18a}
\end{equation}%
\begin{equation}
\left\vert f^{2}\left( \sqrt{ab}\right) -\frac{1}{\ln b-\ln a}%
\dint\limits_{a}^{b}\frac{f(x)}{x}f\left( \frac{ab}{x}\right) dx\right\vert
\leq \frac{\ln \left( \frac{b}{a}\right) }{2}\left( \frac{q-1}{2q-1}\right)
^{1-\frac{1}{q}}H_{3}\left( s,q;h_{3}(\theta ),h_{3}(\vartheta )\right) ,
\label{2-18b}
\end{equation}%
where $M_{1}=\underset{x\in \left[ a,\sqrt{ab}\right] }{\max \left\vert
f(x)\right\vert },\ M_{2}=\underset{x\in \left[ \sqrt{ab},b\right] }{\max
\left\vert f(x)\right\vert },$%
\begin{equation*}
h_{3}(u)=\left\{ 
\begin{array}{c}
1,\ \ \ \ u=1 \\ 
\frac{u-1}{\ln u},\ u\neq 1%
\end{array}%
\right. ,\ u>0
\end{equation*}%
\begin{eqnarray*}
&&H_{3}\left( s,q;h_{3}(\theta ),h_{3}(\vartheta )\right) \\
&=&\left\{ 
\begin{array}{c}
b\left\vert f^{\prime }(b)\right\vert ^{s}M_{1}h_{3}^{1/q}\left( \theta
\right) +a\left\vert f^{\prime }(a)\right\vert ^{s}M_{2}h_{3}^{1/q}\left(
\vartheta \right) , \\ 
\ \left\vert f^{\prime }(a)\right\vert ,\ \left\vert f^{\prime
}(b)\right\vert \leq 1, \\ 
b\left\vert f^{\prime }\left( b\right) \right\vert \left\vert f^{\prime
}\left( a\right) \right\vert ^{1-s}M_{1}h_{3}^{1/q}\left( \theta \right)
+a\left\vert f^{\prime }\left( a\right) \right\vert \left\vert f^{\prime
}\left( b\right) \right\vert ^{1-s}M_{2}h_{3}^{1/q}\left( \vartheta \right) ,
\\ 
\ \left\vert f^{\prime }(a)\right\vert ,\ \left\vert f^{\prime
}(b)\right\vert \geq 1, \\ 
b\left\vert f^{\prime }(b)\right\vert M_{1}h_{3}^{1/q}\left( \theta \right)
+a\left\vert f^{\prime }\left( a\right) \right\vert ^{s}\left\vert f^{\prime
}\left( b\right) \right\vert ^{1-s}M_{2}h_{3}^{1/q}\left( \vartheta \right) ,
\\ 
\ \left\vert f^{\prime }(a)\right\vert \leq 1\leq \ \left\vert f^{\prime
}(b)\right\vert , \\ 
b\left\vert f^{\prime }(b)\right\vert ^{s}M_{1}h_{3}^{1/q}\left( \theta
\right) +a\left\vert f^{\prime }(a)\right\vert M_{2}h_{3}^{1/q}\left(
\vartheta \right) , \\ 
\ \left\vert f^{\prime }(b)\right\vert \leq 1\leq \ \left\vert f^{\prime
}(a)\right\vert .%
\end{array}%
\right. ,
\end{eqnarray*}%
and $\theta $,$\ \vartheta $ are the same as in (\ref{2-6}).
\end{theorem}

\begin{proof}
(1) Let $M_{1}=\underset{x\in \left[ a,\sqrt{ab}\right] }{\max \left\vert
f(x)\right\vert },\ M_{2}=\underset{x\in \left[ \sqrt{ab},b\right] }{\max
\left\vert f(x)\right\vert }.$ Since $\left\vert f^{\prime }\right\vert ^{q}$
is $s$-geometrically convex on $\left[ a,b\right] $, from lemma \ref{2.1}
and H\"{o}lder inequality, we have%
\begin{eqnarray*}
&&\left\vert f(a)f(b)-\frac{1}{\ln b-\ln a}\dint\limits_{a}^{b}\frac{f(x)}{x}%
f\left( \frac{ab}{x}\right) dx\right\vert \\
&\leq &\dint\limits_{0}^{1}\frac{b}{2}\left\vert \ln \left( \frac{a}{b}%
\right) \right\vert \left\vert t-1\right\vert \left( \frac{a}{b}\right) ^{%
\frac{t}{2}}\left\vert f\left( a^{1-t}\left( ab\right) ^{\frac{t}{2}}\right)
\right\vert \left\vert f^{\prime }\left( b^{1-t}\left( ab\right) ^{\frac{t}{2%
}}\right) \right\vert \\
&&+\frac{a}{2}\ln \left( \frac{b}{a}\right) \left\vert t-1\right\vert \left( 
\frac{b}{a}\right) ^{\frac{t}{2}}\left\vert f^{\prime }\left( a^{1-t}\left(
ab\right) ^{\frac{t}{2}}\right) \right\vert \left\vert f\left( b^{1-t}\left(
ab\right) ^{\frac{t}{2}}\right) \right\vert dt
\end{eqnarray*}%
\begin{eqnarray*}
&\leq &\frac{b}{2}\left\vert \ln \left( \frac{a}{b}\right) \right\vert
M_{1}\dint\limits_{0}^{1}\left( 1-t\right) \left( \frac{a}{b}\right) ^{\frac{%
t}{2}}\left\vert f^{\prime }\left( b^{1-t}\left( ab\right) ^{\frac{t}{2}%
}\right) \right\vert dt \\
&&+\frac{a}{2}\ln \left( \frac{b}{a}\right) M_{2}\dint\limits_{0}^{1}\left(
1-t\right) \left( \frac{b}{a}\right) ^{\frac{t}{2}}\left\vert f^{\prime
}\left( a^{1-t}\left( ab\right) ^{\frac{t}{2}}\right) \right\vert dt
\end{eqnarray*}%
\begin{eqnarray*}
&\leq &\frac{b}{2}\left\vert \ln \left( \frac{a}{b}\right) \right\vert
M_{1}\left( \dint\limits_{0}^{1}\left( 1-t\right) ^{\frac{q}{q-1}}dt\right)
^{1-\frac{1}{q}}\left( \dint\limits_{0}^{1}\left( \frac{a}{b}\right) ^{\frac{%
qt}{2}}\left\vert f^{\prime }\left( b^{1-t}\left( ab\right) ^{\frac{t}{2}%
}\right) \right\vert ^{q}dt\right) ^{\frac{1}{q}} \\
&&+\frac{a}{2}\ln \left( \frac{b}{a}\right) M_{2}\left(
\dint\limits_{0}^{1}\left( 1-t\right) ^{\frac{q}{q-1}}dt\right) ^{1-\frac{1}{%
q}}\left( \dint\limits_{0}^{1}\left( \frac{b}{a}\right) ^{\frac{qt}{2}%
}\left\vert f^{\prime }\left( a^{1-t}\left( ab\right) ^{\frac{t}{2}}\right)
\right\vert ^{q}dt\right) ^{\frac{1}{q}}
\end{eqnarray*}%
\begin{eqnarray}
&\leq &\frac{b\ln \left( \frac{b}{a}\right) M_{1}}{2}\left( \frac{q-1}{2q-1}%
\right) ^{1-\frac{1}{q}}\left( \dint\limits_{0}^{1}\left( \frac{a}{b}\right)
^{\frac{qt}{2}}\left\vert f^{\prime }\left( a\right) \right\vert ^{q\left(
t/2\right) ^{s}}\left\vert f^{\prime }\left( b\right) \right\vert ^{q\left(
\left( 2-t\right) /2\right) ^{s}}dt\right) ^{\frac{1}{q}}  \notag \\
&&+\frac{a\ln \left( \frac{b}{a}\right) M_{2}}{2}\left( \frac{q-1}{2q-1}%
\right) ^{1-\frac{1}{q}}\left( \dint\limits_{0}^{1}\left( \frac{b}{a}\right)
^{\frac{qt}{2}}\left\vert f^{\prime }\left( b\right) \right\vert ^{q\left(
t/2\right) ^{s}}\left\vert f^{\prime }\left( a\right) \right\vert ^{q\left(
\left( 2-t\right) /2\right) ^{s}}dt\right) ^{\frac{1}{q}}.  \label{2-18}
\end{eqnarray}

(i) \ If $1\geq \left\vert f^{\prime }(a)\right\vert ,\ \left\vert f^{\prime
}(b)\right\vert ,$ by (\ref{2-a}) we have%
\begin{equation*}
\dint\limits_{0}^{1}\left( \frac{a}{b}\right) ^{\frac{qt}{2}}\left\vert
f^{\prime }\left( a\right) \right\vert ^{q\left( t/2\right) ^{s}}\left\vert
f^{\prime }\left( b\right) \right\vert ^{q\left( \left( 2-t\right) /2\right)
^{s}}dt=\left\vert f^{\prime }\left( b\right) \right\vert ^{qs}h_{3}\left(
\theta \right) ,
\end{equation*}%
\begin{equation}
\dint\limits_{0}^{1}\left( \frac{b}{a}\right) ^{\frac{qt}{2}}\left\vert
f^{\prime }\left( b\right) \right\vert ^{q\left( t/2\right) ^{s}}\left\vert
f^{\prime }\left( a\right) \right\vert ^{q\left( \left( 2-t\right) /2\right)
^{s}}dt=\left\vert f^{\prime }\left( a\right) \right\vert ^{qs}h_{3}\left(
\vartheta \right) .  \label{2-19}
\end{equation}

(ii) \ If $1\leq \left\vert f^{\prime }(a)\right\vert ,\ \left\vert
f^{\prime }(b)\right\vert ,$ by (\ref{2-a}) we have%
\begin{equation*}
\dint\limits_{0}^{1}\left( \frac{a}{b}\right) ^{\frac{qt}{2}}\left\vert
f^{\prime }\left( a\right) \right\vert ^{q\left( t/2\right) ^{s}}\left\vert
f^{\prime }\left( b\right) \right\vert ^{q\left( \left( 2-t\right) /2\right)
^{s}}dt\leq \left( \left\vert f^{\prime }\left( b\right) \right\vert
\left\vert f^{\prime }\left( a\right) \right\vert ^{1-s}\right)
^{q}h_{3}\left( \theta \right) ,
\end{equation*}%
\begin{equation}
\dint\limits_{0}^{1}\left( \frac{b}{a}\right) ^{\frac{qt}{2}}\left\vert
f^{\prime }\left( b\right) \right\vert ^{q\left( t/2\right) ^{s}}\left\vert
f^{\prime }\left( a\right) \right\vert ^{q\left( \left( 2-t\right) /2\right)
^{s}}dt\leq \left( \left\vert f^{\prime }\left( a\right) \right\vert
\left\vert f^{\prime }\left( b\right) \right\vert ^{1-s}\right)
^{q}h_{3}\left( \vartheta \right) .  \label{2-20}
\end{equation}%
(iii) \ If $\left\vert f^{\prime }(a)\right\vert \leq 1\leq \ \left\vert
f^{\prime }(b)\right\vert ,$ by (\ref{2-a}) we obtain that%
\begin{equation*}
\dint\limits_{0}^{1}\left( \frac{a}{b}\right) ^{\frac{qt}{2}}\left\vert
f^{\prime }\left( a\right) \right\vert ^{q\left( t/2\right) ^{s}}\left\vert
f^{\prime }\left( b\right) \right\vert ^{q\left( \left( 2-t\right) /2\right)
^{s}}dt\leq \left\vert f^{\prime }\left( b\right) \right\vert
^{q}h_{3}\left( \theta \right) ,
\end{equation*}%
\begin{equation}
\dint\limits_{0}^{1}\left( \frac{b}{a}\right) ^{\frac{qt}{2}}\left\vert
f^{\prime }\left( b\right) \right\vert ^{q\left( t/2\right) ^{s}}\left\vert
f^{\prime }\left( a\right) \right\vert ^{q\left( \left( 2-t\right) /2\right)
^{s}}dt\leq \left( \left\vert f^{\prime }\left( a\right) \right\vert
^{s}\left\vert f^{\prime }\left( b\right) \right\vert ^{1-s}\right)
^{q}h_{3}\left( \vartheta \right) .  \label{2-21}
\end{equation}%
(iv) \ If $\left\vert f^{\prime }(b)\right\vert \leq 1\leq \ \left\vert
f^{\prime }(a)\right\vert ,$ by (\ref{2-a}) we obtain that%
\begin{equation*}
\dint\limits_{0}^{1}\left( \frac{a}{b}\right) ^{\frac{qt}{2}}\left\vert
f^{\prime }\left( a\right) \right\vert ^{q\left( t/2\right) ^{s}}\left\vert
f^{\prime }\left( b\right) \right\vert ^{q\left( \left( 2-t\right) /2\right)
^{s}}dt\leq \left( \left\vert f^{\prime }\left( a\right) \right\vert
^{1-s}\left\vert f^{\prime }\left( b\right) \right\vert ^{s}\right)
^{q}h_{3}\left( \theta \right) ,
\end{equation*}%
\begin{equation}
\dint\limits_{0}^{1}\left( \frac{b}{a}\right) ^{\frac{qt}{2}}\left\vert
f^{\prime }\left( b\right) \right\vert ^{q\left( t/2\right) ^{s}}\left\vert
f^{\prime }\left( a\right) \right\vert ^{q\left( \left( 2-t\right) /2\right)
^{s}}dt\leq \left\vert f^{\prime }\left( a\right) \right\vert
^{q}h_{3}\left( \vartheta \right) .  \label{2-22}
\end{equation}%
From (\ref{2-18}) to (\ref{2-22}), (\ref{2-18a}) holds.

(2) Let $M_{1}=\underset{x\in \left[ a,\sqrt{ab}\right] }{\max \left\vert
f(x)\right\vert },\ M_{2}=\underset{x\in \left[ \sqrt{ab},b\right] }{\max
\left\vert f(x)\right\vert }.$ Since $\left\vert f^{\prime }\right\vert ^{q}$
is $s$-geometrically convex on $\left[ a,b\right] $, from lemma \ref{2.1}
and H\"{o}lder inequality, we have%
\begin{equation*}
\left\vert f^{2}\left( \sqrt{ab}\right) -\frac{1}{\ln b-\ln a}%
\dint\limits_{a}^{b}\frac{f(x)}{x}f\left( \frac{ab}{x}\right) dx\right\vert
\end{equation*}%
\begin{eqnarray}
&\leq &\frac{b}{2}\ln \left( \frac{b}{a}\right) M_{1}\left( \frac{q-1}{2q-1}%
\right) ^{1-\frac{1}{q}}\left( \dint\limits_{0}^{1}\left( \frac{a}{b}\right)
^{\frac{qt}{2}}\left\vert f^{\prime }\left( a\right) \right\vert ^{q\left(
t/2\right) ^{s}}\left\vert f^{\prime }\left( b\right) \right\vert ^{q\left(
\left( 2-t\right) /2\right) ^{s}}dt\right) ^{\frac{1}{q}}  \notag \\
&&+\frac{a}{2}\ln \left( \frac{b}{a}\right) M_{2}\left( \frac{q-1}{2q-1}%
\right) ^{1-\frac{1}{q}}\left( \dint\limits_{0}^{1}\left( \frac{b}{a}\right)
^{\frac{qt}{2}}\left\vert f^{\prime }\left( b\right) \right\vert ^{q\left(
t/2\right) ^{s}}\left\vert f^{\prime }\left( a\right) \right\vert ^{q\left(
\left( 2-t\right) /2\right) ^{s}}dt\right) ^{\frac{1}{q}}.  \label{2-23}
\end{eqnarray}%
From (\ref{2-23}) and (\ref{2-19}) to (\ref{2-22}), (\ref{2-18b}) holds.
\end{proof}

If taking $s=1$ in Theorem \ref{2.3}, we can derive the following corollary.

\begin{corollary}
Let $f:I\subseteq \mathbb{R}_{+}\mathbb{\rightarrow R}_{+}$ be
differentiable on $I^{\circ }$, and $a,b\in I^{\circ }$ with $a<b$ and $%
f^{\prime }\in L\left[ a,b\right] .$ If $\left\vert f^{\prime }\right\vert
^{q}$ is geometrically convex on $\left[ a,b\right] $ for $q>1,$ then%
\begin{equation*}
\left\vert f(a)f(b)-\frac{1}{\ln b-\ln a}\dint\limits_{a}^{b}\frac{f(x)}{x}%
f\left( \frac{ab}{x}\right) dx\right\vert \leq \frac{1}{2}\ln \left( \frac{b%
}{a}\right) \left( \frac{q-1}{2q-1}\right) ^{1-\frac{1}{q}}H_{3}\left(
1,q;h_{3}(\theta ),h_{3}(\vartheta )\right) ,
\end{equation*}%
\begin{equation*}
\left\vert f^{2}\left( \sqrt{ab}\right) -\frac{1}{\ln b-\ln a}%
\dint\limits_{a}^{b}\frac{f(x)}{x}f\left( \frac{ab}{x}\right) dx\right\vert
\leq \frac{1}{2}\ln \left( \frac{b}{a}\right) \left( \frac{q-1}{2q-1}\right)
^{1-\frac{1}{q}}H_{3}\left( 1,q;h_{3}(\theta ),h_{3}(\vartheta )\right) ,
\end{equation*}%
where $\theta ,\ \vartheta $, $H_{3}$ and $h_{3}$ are the same as in Theorem %
\ref{2.3}.
\end{corollary}

\section{Application to Special Means}

Let us recall the following special means of two nonnegative number $a,b$
with $b>a:$

\begin{enumerate}
\item The arithmetic mean%
\begin{equation*}
A=A\left( a,b\right) :=\frac{a+b}{2}.
\end{equation*}

\item The geometric mean%
\begin{equation*}
G=G\left( a,b\right) :=\frac{a+b}{2}.
\end{equation*}

\item The logarithmic mean%
\begin{equation*}
L=L\left( a,b\right) :=\frac{b-a}{\ln b-\ln a}.
\end{equation*}

\item The p-logarithmic mean%
\begin{equation*}
L_{p}=L_{p}\left( a,b\right) :=\left( \frac{b^{p+1}-a^{p+1}}{(p+1)(b-a)}%
\right) ^{\frac{1}{p}},\ \ p\in 
%TCIMACRO{\U{211d} }%
%BeginExpansion
\mathbb{R}
%EndExpansion
\backslash \left\{ -1,0\right\} .
\end{equation*}
\end{enumerate}

Let $f(x)=\left( x^{s}/s\right) +1,\ x\in \left( 0,1\right] ,\ 0<s<1,\ q\geq
1,$ and then the function $\left\vert f^{\prime }(x)\right\vert
^{q}=x^{(s-1)q}$ is $s$-geometrically convex on $\left( 0,1\right] $ for $%
0<s<1$ (see \cite{ZJQ12}).

\begin{proposition}
\label{3.1}Let $0<a<b\leq 1,\ 0<s<1,$and $q\geq 1.$ Then 
\begin{eqnarray*}
&&\left\vert G^{2}\left( \frac{a^{s}}{s}+1,\frac{b^{s}}{s}+1\right) -\frac{2%
}{s^{2}}A\left( G^{2}\left( a^{s},b^{s}\right) ,s^{2}\right) -\frac{2}{s}%
L_{s-1}^{s-1}\left( a,b\right) L\left( a,b\right) \right\vert \\
&\leq &\frac{1}{G^{(s-1)^{2}}(a,b)}\left( \frac{1}{\left( s^{2}-s+1\right) q}%
\right) ^{\frac{1}{q}}\left( \frac{b-a}{4L\left( a,b\right) }\right) ^{1-%
\frac{1}{q}} \\
&&\times \left[ M_{1}G\left( a^{-(s-1)^{2}},b^{s}\right) \left\{ b^{\frac{%
\left( s^{2}-s+1\right) q}{2}}-L(a^{\frac{\left( s^{2}-s+1\right) q}{2}},b^{%
\frac{\left( s^{2}-s+1\right) q}{2}})\right\} ^{\frac{1}{q}}\right. \\
&&\left. M_{2}G\left( b^{-(s-1)^{2}},a^{s}\right) \left\{ L(a^{\frac{\left(
s^{2}-s+1\right) q}{2}},b^{\frac{\left( s^{2}-s+1\right) q}{2}})-a^{\frac{%
\left( s^{2}-s+1\right) q}{2}}\right\} ^{\frac{1}{q}}\right] ,
\end{eqnarray*}%
\begin{eqnarray*}
&&\left\vert \left( \frac{G^{s}\left( a,b\right) }{s}+1\right) ^{2}-\frac{2}{%
s^{s}}A\left( G^{2}\left( a^{s},b^{s}\right) ,s^{2}\right) -\frac{2}{s}%
L_{s-1}^{s-1}\left( a,b\right) L\left( a,b\right) \right\vert \\
&\leq &\frac{1}{G^{(s-1)^{2}}(a,b)}\left( \frac{1}{\left( s^{2}-s+1\right) q}%
\right) ^{\frac{1}{q}}\left( \frac{b-a}{4L\left( a,b\right) }\right) ^{1-%
\frac{1}{q}} \\
&&\times \left[ M_{1}G\left( a^{-(s-1)^{2}},b^{s}\right) \left\{ L(a^{\frac{%
\left( s^{2}-s+1\right) q}{2}},b^{\frac{\left( s^{2}-s+1\right) q}{2}})-a^{%
\frac{\left( s^{2}-s+1\right) q}{2}}\right\} ^{\frac{1}{q}}\right. \\
&&\left. M_{2}G\left( b^{-(s-1)^{2}},a^{s}\right) \right] \left\{ b^{\frac{%
\left( s^{2}-s+1\right) q}{2}}-L(a^{\frac{\left( s^{2}-s+1\right) q}{2}},b^{%
\frac{\left( s^{2}-s+1\right) q}{2}})\right\} ^{\frac{1}{q}}.
\end{eqnarray*}%
where \bigskip $M_{1}=\left( \sqrt{ab}^{s}/s\right) +1$ and $\ M_{2}=\left(
b^{s}/s\right) +1.$
\end{proposition}

\begin{proof}
Let $f(x)=\left( x^{s}/s\right) +1,\ x\in \left( 0,1\right] ,\ 0<s<1.$ Then $%
\left\vert f^{\prime }(a)\right\vert =a^{s-1}>b^{s-1}=\left\vert f^{\prime
}(b)\right\vert \geq 1,$ $M_{1}=\underset{x\in \left[ a,\sqrt{ab}\right] }{%
\max \left\vert f(x)\right\vert =}\left( \sqrt{ab}^{s}/s\right) +1,$ $\
M_{2}=\underset{x\in \left[ \sqrt{ab},b\right] }{\max \left\vert
f(x)\right\vert }=\left( b^{s}/s\right) +1.$ Thus, by Theorem \ref{2.2},
Proposition \ref{3.1} is proved.
\end{proof}

\begin{proposition}
\label{3.2}Let $0<a<b\leq 1,\ 0<s<1,$and $q>1.$ Then 
\begin{eqnarray*}
&&\left\vert G^{2}\left( \frac{a^{s}}{s}+1,\frac{b^{s}}{s}+1\right) -\frac{2%
}{s^{2}}A\left( G^{2}\left( a^{s},b^{s}\right) ,s^{2}\right) -\frac{2}{s}%
L_{s-1}^{s-1}\left( a,b\right) L\left( a,b\right) \right\vert \\
&\leq &\frac{b-a}{2L(a,b)}\left( \frac{q-1}{2q-1}\right) ^{1-\frac{1}{q}}%
\frac{L^{\frac{1}{q}}(a^{\frac{\left( s^{2}-s+1\right) q}{2}},b^{\frac{%
\left( s^{2}-s+1\right) q}{2}})}{G^{(s-1)^{2}}(a,b)} \\
&&\times \left\{ M_{1}G\left( a^{-(s-1)^{2}},b^{s}\right) +M_{2}G\left(
b^{-(s-1)^{2}},a^{s}\right) \right\}
\end{eqnarray*}%
\begin{eqnarray*}
&&\left\vert \left( \frac{G^{s}\left( a,b\right) }{s}+1\right) ^{2}-\frac{2}{%
s^{s}}A\left( G^{2}\left( a^{s},b^{s}\right) ,s^{2}\right) -\frac{2}{s}%
L_{s-1}^{s-1}\left( a,b\right) L\left( a,b\right) \right\vert \\
&\leq &\frac{b-a}{2L(a,b)}\left( \frac{q-1}{2q-1}\right) ^{1-\frac{1}{q}}%
\frac{L^{\frac{1}{q}}(a^{\frac{\left( s^{2}-s+1\right) q}{2}},b^{\frac{%
\left( s^{2}-s+1\right) q}{2}})}{G^{(s-1)^{2}}(a,b)} \\
&&\times \left\{ M_{1}G\left( a^{-(s-1)^{2}},b^{s}\right) +M_{2}G\left(
b^{-(s-1)^{2}},a^{s}\right) \right\}
\end{eqnarray*}%
where \bigskip $M_{1}=\left( \sqrt{ab}^{s}/s\right) +1$ and $\ M_{2}=\left(
b^{s}/s\right) +1.$
\end{proposition}

\begin{proof}
Let $f(x)=\left( x^{s}/s\right) +1,\ x\in \left( 0,1\right] ,\ 0<s<1.$ Then $%
\left\vert f^{\prime }(a)\right\vert =a^{s-1}>b^{s-1}=\left\vert f^{\prime
}(b)\right\vert \geq 1,$ $M_{1}=\underset{x\in \left[ a,\sqrt{ab}\right] }{%
\max \left\vert f(x)\right\vert =}\left( \sqrt{ab}^{s}/s\right) +1,$ $\
M_{2}=\underset{x\in \left[ \sqrt{ab},b\right] }{\max \left\vert
f(x)\right\vert }=\left( b^{s}/s\right) +1.$ Thus, by Theorem \ref{2.3},
Proposition \ref{3.2} is proved.
\end{proof}

\end{document}